\documentclass[10pt]{amsart}%
\usepackage{graphicx}
\usepackage{amscd, color}
\usepackage{amsmath}
\usepackage{amsfonts}
\usepackage{amssymb}

\providecommand{\U}[1]{\protect\rule{.1in}{.1in}}
\providecommand{\U}[1]{\protect\rule{.1in}{.1in}}
\providecommand{\U}[1]{\protect\rule{.1in}{.1in}} \textwidth 15.8cm
\theoremstyle{plain}

\newtheorem{theorem}{Theorem}[section]
\newtheorem{proposition}[theorem]{Proposition}
\newtheorem{corollary}[theorem]{Corollary}

\newtheorem{definition}[theorem]{Definition}
\numberwithin{equation}{section}

\begin{document}
\title[An abstract result on Cohen strongly summing operators]{An abstract result on Cohen strongly summing operators}
\author{Jamilson R. Campos}
\address{Departamento de Ci\^{e}ncias Exatas \\
\indent Universidade Federal da Para\'{\i}ba - Campus IV \\
\indent R. da Mangueira, s/n, Centro \\
\indent Rio Tinto, 58.297-000, Brazil.}
\email{jamilson@dce.ufpb.br and jamilsonrc@gmail.com}

\begin{abstract}
We present an abstract result that characterizes the coincidence of certain classes of linear operators with the class of Cohen strongly summing linear operators. Our argument is extended to multilinear operators and, as a consequence, we establish a few alternative characterizations for the class of Cohen strongly summing multilinear operators.
\end{abstract}
\maketitle

\section{Introduction}

With the work of A. Dvoretzky and C. Rogers \cite{dvoretzky50}, it became established that in infinite dimensional spaces always exists unconditionally convergent series which does not converge absolutely. So, A. Grothendieck \cite{grothen56} introduces the conceptual essence of absolutely summing operators as those that improve the convergence of series, towards transforming an unconditionally convergent series in an absolutely convergent one. After that, many research efforts are aimed to generalize the concept of absolutely summing operators and make Grothendieck's ideas clearer, such as works of B. Mitiagin and A. Pelczy\'{n}ski \cite{mitiagin66} and J. Lindenstrauss and A. Pelczy\'{n}ski \cite{linden68}. We recommend the paper \cite{daniel11} for a panorama on absolutely summing operators.

In 1967, Pietsch (\cite{pietsch67}, pág. 338) shows that the identity operator from $l_{1}$ into $l_{2}$ is absolutely $2$-summing while its conjugate, from $l_{2}$ into $l_{\infty}$, is not absolutely $2$-summing. Motivated by this fact, J. S. Cohen \cite{cohen73} introduces the class of strongly $p$-summing linear operators which characterizes the conjugate of the class of absolutely $p^*$-summing linear operators, with $1/p+1/p^{\ast}=1$. In his work, Cohen proves several results for this new class such as inclusion relations and a Pietsch domination type theorem.

The concept of Cohen strongly summing  multilinear operator was introduced and studied by D. Achour and L. Mezrag \cite{achour07} and Mezrag and K. Saadi \cite{mezrag09} established some inclusion results between the class of Cohen strongly summing  multilinear operators and other classes of operators such as the class of multiple summing operators. A related concept and new generalizations of concept of Cohen strongly summing  multilinear operators have been recently studied, such as the classes of almost Cohen strongly multilinear operators \cite{bu12} and the class of multiple Cohen strongly summing multilinear operators \cite{jamilson13}.

In this paper, we prove an abstract result derived from the Full General Pietsch Domination Theorem (\cite{daniel12}, Theorem 4.6 page 1256) which has an immediate application regarding the class of Cohen strongly summing linear operators. This result is also extended to the multilinear case and with that we can establish alternative definitions for Cohen strongly summing  multilinear operators. Through these definitions, it is possible to characterize the class of Cohen strongly summing  multilinear operators by means of arbitrary sequences.

\vspace{0,25cm}
\section{Notation and background }\label{bn}

\vspace{0,25cm}
Let $n \in \mathbb{N}$, $E,E_1,...,E_n$ and $F$ be Banach spaces over $\mathbb{K} = \mathbb{R}$ or $\mathbb{C}$. The space of all continuous linear operators from $E$ into $F$ will be represented by $\mathcal{L}(E;F)$. We will denote by $\mathcal{L}(E_1,...,E_n;F)$ the space of all continuous $n$-linear operators from $E_1,...,E_n$ into $F$ and if $E_1= \cdots =E_n$ we just write $\mathcal{L}(^nE;F)$.

If $1 \leq p \leq \infty$, we will denote by $l_p(E)$ and $l_p^w(E)$ the spaces of all sequences  $(x_i)_{i=1}^\infty$ in $E$ with the norms
\[||(x_i)_{i=1}^\infty||_p = \left( \sum_{i=1}^\infty ||x_i||^p\right)^{1/p} < \infty \ \ \ \mathrm{and}\ \ \ ||(x_i)_{i=1}^\infty||_{w,p} = \sup_{\psi \in B_{E^{'}}}\left( \sum_{i=1}^\infty |\psi(x_i)|^p\right)^{1/p} < \infty \ ,\]
respectively, where $E^{'}$ represents the topological dual of $E$ and $B_{E}$ denotes the closed unit ball of $E$.

\vspace{0,25cm}
\begin{definition}
[Cohen, \cite{cohen73}]\label{defcohen} A sequence $(x_{i})_{i=1}^{\infty}$ in
a Banach space $E$ is Cohen strongly $p$-summing if the series $\sum
_{i=1}^{\infty}\varphi_{i}(x_{i})$ converges for all $(\varphi_{i}%
)_{i=1}^{\infty}\in l_{p^{*}}^{w}(E^{^{\prime}})$, with $1/p+1/p^{*}=1$.
\end{definition}

\vspace{0,25cm}
We denote by $l_{p}\langle E\rangle$ the space of Cohen strongly $p$-summing
sequences in $E$. It is possible, for reasons of management, to replace the
series $\sum_{i=1}^{\infty}\varphi_{i}(x_{i})$ in Definition \ref{defcohen} by
the series $\sum_{i=1}^{\infty}|\varphi_{i}(x_{i})|$, which we use in our
text.

The space $l_{p}\langle E\rangle$ is a Banach space (\cite{khalil82}, page 223) with
the norm
\[
||(x_{i})_{i=1}^{\infty}||_{C,p} = \sup_{||(\varphi_{i})_{i=1}^{\infty
}||_{w,p^{*}} \leq1}\sum_{i=1}^{\infty}|\varphi_{i}(x_{i})|\ .
\]

\vspace{0,25cm}
\begin{definition}
[Cohen, \cite{cohen73}]
Let $1 < p \leq \infty$. An operator $T \in \mathcal{L}(E;F)$ is Cohen strongly $p$-summing if there exists a constant  $C>0$ such that for all $m \in \mathbb{N},\ x_i \in E$ and $\varphi_i \in F^{'},\ i=1,...,m$,
\begin{equation}\label{def3}
\sum_{i=1}^m |\varphi_i(T(x_i))| \leq C\, ||(x_i)_{i=1}^m||_p ||(\varphi_i)_{i=1}^m||_{w,p^*}.
\end{equation}
\end{definition}
In an equivalent manner, an operator $T \in \mathcal{L}(E;F)$ is Cohen strongly $p$-summing if the operator
\begin{equation*}
\widehat{T}\  :l_{p}\left(  E\right)  \rightarrow l_p\langle F\rangle \ ; \left(x_{i}\right)_{i=1}^{\infty} \mapsto \left(T\left(x_{i}\right)\right)_{i=1}^{\infty}
\end{equation*}
is well-defined and continuous. The space of all Cohen strongly $p$-summing linear operators will be denoted by $\mathcal{D}_p(E;F)$. The smallest $C$ such that (\ref{def3}) is satisfied defines a norm on $\mathcal{D}_p(E;F)$, denoted by $d_p(\cdot)$.

\vspace{0,25cm}
\begin{definition}[Achour-Mezrag, \cite{achour07}]\label{defAchour}
Let $1 < p \leq\infty$, $E_{j},F$ be Banach spaces, $j=1,...,n$. An continuous $n$-linear operator $T:E_{1}
\times... \times E_{n} \rightarrow F$ is Cohen strongly $p$-summing if there exists a constant $C>0$ such that for all $x_{1}^{(j)},...,x_{m}^{(j)} \in E_{j}$ and $\varphi_{1},...,\varphi_{m} \in
F^{^{\prime}}$ and for all $m\in\mathbb{N}$,
\begin{equation}
\label{cohenm}\sum_{i=1}^{m} \left\vert \varphi_{i}\left(T\left(x_{i}^{(1)},...,x_{i}^{(n)}\right)\right)\right\vert \leq C
\left(  \sum_{i=1}^{m} \prod_{j=1}^{n} \left\Vert x_{i}^{(j)}\right\Vert^{p} \right)
^{1/p}\ ||(\varphi_{i})_{i=1}^{m}||_{w,p^{*}} \ .
\end{equation}
\end{definition}

The space of all Cohen strongly $p$-summing multilinear operators will be represented by \\ $\mathcal{L}_{Coh,p}(E_1,...,E_n;F)$. The smallest $C$ such that (\ref{cohenm}) is satisfied defines a norm on $\mathcal{L}_{Coh,p}(E_1,...,E_n;F)$, denoted by $||\cdot||_{Coh,p}$.

\vspace{0,25cm}
An abstract very useful tool developed by Pellegrino \emph{et al}. \cite{daniel12}, the Full General Pietsch Domination Theorem, allows to establish in a very simplified way various Pietsch Domination-type theorems for many classes of operators. This approach was initiated in the works of G. Botelho, D. Pellegrino and P. Rueda \cite{BPRn} and D. Pellegrino and J. Santos \cite{jmaa1}.

\vspace{0,25cm}
Let $X_{1},...,X_{n},Y$ and $E_{1},...,E_{r}$ be arbitrary non-empty sets, $\mathcal{H}$ be a family of operators from $X_{1} \times \cdots\times X_{n}$ into $Y$. Let $K_{1},...,K_{t}$ be compact Hausdorff topological spaces, $G_{1},...,G_{t}$ be Banach spaces and suppose that the mappings
\[
\left\{
\begin{array}
[c]{l}%
R_{j}:\ K_{j} \times E_{1} \times\cdots\times E_{r} \times G_{j}
\rightarrow[0,\infty)\ , \ j=1,...,t\ ,\\
S:\ \mathcal{H} \times E_{1} \times\cdots\times E_{r} \times G_{1}
\times\cdots\times G_{t} \rightarrow[0,\infty)\\
\end{array}
\right.
\]
have the following properties:

\begin{quote}
\noindent $1)$ For each $x^{(l)} \in E_{l}$ and $b \in G_{j}$, with $(j,l)
\in\{1,...,t\} \times\{1,...,r\}$, the mapping
\[
(R_{j})_{x^{(1)},...,x^{(r)},b}\ : \ K_{j} \rightarrow[0,\infty)\ ,
\]
defined by $(R_{j})_{x^{(1)},...,x^{(r)},b}(\varphi) = R_{j}\left(\varphi
,x^{(1)},...,x^{(r)},b\right)$, is continuous;

\noindent $2)$ The following inequalities hold:
\[
\left\{
\begin{array}
[c]{l}%
R_{j}\left(\varphi,x^{(1)},...,x^{(r)},\eta_{j} b^{(j)}\right) \leq\eta_{j} R_{j}%
\left(\varphi,x^{(1)},...,x^{(r)}, b^{(j)}\right)\\
S\left(f, x^{(1)},...,x^{(r)}, \alpha_{1} b^{(1)},...,\alpha_{t} b^{(t)}\right)
\geq\alpha_{1}...\alpha_{t} S\left(f, x^{(1)},...,x^{(r)}, b^{(1)},..., b^{(t)}\right),\\
\end{array}
\right.
\]
for all $\varphi\in K_{j},\ x^{(l)} \in E_{l}$ (with $l=1,...,r$), $0 \leq
\eta_{j}, \alpha_{j} \leq1,\ b^{(j)} \in G_{j}$, $j=1,...,t$ and $f
\in\mathcal{H}$.
\end{quote}

\vspace{0,25cm}
Under these conditions, we have the following definition and theorem obtained from \cite{daniel12}:

\vspace{0,25cm}
\begin{definition}[Pellegrino \emph{et al}. \cite{daniel12}]\label{defabst}
Let $0<p_{1},...,p_{t},p_{0}<\infty$, with $\frac{1}%
{p_{0}}=\frac{1}{p_{1}}+\cdots+\frac{1}{p_{t}}$. An application $f\,:\,X_{1}%
\times\cdots\times X_{n}\rightarrow Y$ belonging to $\mathcal{H}$ is $R_{1},...,R_{t}$%
-$S$-abstract $(p_{1},...,p_{t})$-summing if there is a constant $C>0$ such that
\begin{equation*}
\left(  \sum_{i=1}^{m}\left(  S\left(f,x_{i}^{(1)},...,x_{i}^{(r)},b_{i}%
^{(1)},...,b_{i}^{(t)}\right)\right)  ^{p_{0}}\right)  ^{1/p_{0}}
\leq C\prod _{k=1}^{t}\sup_{\varphi\in K_{k}}\left(  \sum_{i=1}^{m}R_{k}\left(\varphi
,x_{i}^{(1)},...,x_{i}^{(r)},b_{i}^{(k)}\right)^{p_{k}}\right)  ^{1/p_{k}},
\end{equation*}
for all $x_{1}^{(s)},...,x_{m}^{(s)}\in E_{s},\ b_{1}^{(s)},...,b_{m}^{(s)}\in
G_{l},\ m\in\mathbb{N}$ and $(s,l)\in\{1,...,r\}\times\{1,...,t\}$.
\end{definition}

\vspace{0,25cm}
\begin{theorem}[Full General Pietsch Domination Theorem, Pellegrino \emph{et al}. (\cite{daniel12})]\label{tdpg}
An application $f\in\mathcal{H}$ is $R_{1},...,R_{t}$%
-$S$-abstract $(p_{1},...,p_{t})$-summing if and only if there exist a constant $C>0$ and Borel probability measures $\mu_{k}$ in $K_{k},\ k=1,...,t$, such that
\[
S\left(f,x^{(1)},...,x^{(r)},b^{(1)},...,b^{(t)}\right)\leq C\prod_{k=1}^{t}\left(
\int_{K_{k}}R_{k}\left(\varphi,x^{(1)},...,x^{(r)},b^{(k)}\right)^{p_{k}}d\mu_{k}\right)
^{1/p_{k}}\ ,
\]
for all $x^{(l)}\in E_{l},\ l=1,...,r$ and $b^{(k)}\in G_{k}$, with $k=1,...,t$.
\end{theorem}

\vspace{0,25cm}
\section{The linear case}

\vspace{0,25cm}
We show that the class of Cohen strongly $p$-summing linear operators can be characterized by means of different inequalities. This means a coincidence result between classes of linear operators.

\vspace{0,25cm}
Let $p^* \in \left(  1,\infty\right)  $ with $1 = \frac{1}{p} + \frac{1}{p^{\ast}}$ and
\[
\Gamma=\left\{  \left(  q_0,q_1\right)  \in \left.\left[ 1,\infty \right.\right) \times  \left( 1,\infty\right) :\frac
{1}{q_0}= \frac{1}{q_1} + \frac{1}{p^{\ast}}\right\}  .
\]
We will denote by $\mathcal{C}_{(q_0,q_1;p)}(E;F)$ the class of all operators $T\in\mathcal{L}\left(  E;F\right)  $
such that exists a constant $C > 0$ satisfying
\[
\left(\sum\limits_{j=1}^{m} \left\vert \varphi_{i}\left(  T\left(  x_{i}\right)  \right)  \right\vert
^{q_0}\right)  ^{1/q_0}\leq C\left\Vert \left(  x_{i}\right)  _{i=1}
^{m}\right\Vert _{q_1}\left\Vert \left(  \varphi_{i}\right)  _{i=1}
^{m}\right\Vert _{w,p^{\ast}},
\]
for all positive integers $m$ and all $x_i \in E$, $\varphi_i \in F^{'}$, $i=1,...,m$.

It follows immediately from Definition \ref{defcohen} that
\begin{equation*}
\mathcal{D}_{p}(E;F) \subset \mathcal{C}_{(q_0,q_1;p)}(E;F),
\end{equation*}
for all $\left(  q_0,q_1\right) \in \Gamma$.

Another noteworthy fact is that the class $\mathcal{C}_{(q_0,q_1;p)}(E;F)$ is trivial if $p<q_1$. The proof of this fact is given by the following proposition:

\vspace{0,25cm}
\begin{proposition}\label{obscohgen}
Let $1 < p < \infty$, with $1 = 1/p + 1/p^*$, and $(q_0,q_1) \in \Gamma$, that is $1/q_0 = 1/q_1 + 1/p^*$. If $p<q_1$, then $\mathcal{C}_{(q_0,q_1;p)}(E;F) = \{ 0 \}$.
\end{proposition}

\begin{proof}
The case $E=\{0\}$ or $F=\{0\}$ is immediate. Let us suppose that $E\neq \{0\}$ and $F \neq \{0\}$. Let \begin{equation}\label{equa1}
(\lambda_i)_{i=1}^\infty = (\alpha_i \beta_i)_{i=1}^\infty \notin l_{q_0}\ \  \mathrm{with}\ \ (\alpha_i)_{i=1}^\infty \in l_{p^*}\ \ \mathrm{and}\ \ (\beta_i)_{i=1}^\infty \in l_{q_1}.
\end{equation}

If $T \in \mathcal{C}_{(q_0,q_1;p)}(E;F)$, $0 \neq x \in E$ and $0 \neq \varphi \in F^{'}$, then, for all $m \in \mathbb{N}$,
\begin{align*}
\left(\sum_{i=1}^m |\varphi(T(\lambda_ix))|^{q_0}\right)^{1/q_0} & = \left(\sum_{i=1}^m |\varphi(T(\alpha_i\beta_ix))|^{q_0}\right)^{1/q_0} \\
& = \left(\sum_{i=1}^m |(\alpha_i\varphi)(T(\beta_ix))|^{q_0}\right)^{1/q_0} \\
& \leq C ||(\beta_ix)_{i=1}^m||_{q_1} \, ||(\alpha_i\varphi)_{i=1}^m||_{w,p*} \ .
\end{align*}
Thus,
\begin{align*}
|\varphi(T(x))| \left(\sum_{i=1}^m |\lambda_i|^{q_0}\right)^{1/q_0} & \leq C ||x||\, ||(\beta_i)_{i=1}^m||_{q_1} \sup_{\psi \in B_{F^{''}}} \left( \sum_{i=1}^m |\psi(\alpha_i\varphi)|^{p^*} \right)^{1/p^*} \\
& = C ||x||\, ||(\beta_i)_{i=1}^m||_{q_1} \sup_{\psi \in B_{F^{''}}} |\psi(\varphi)| \left( \sum_{i=1}^m |\alpha_i|^{p^*} \right)^{1/p^*} \\
&  = C ||x||\, ||(\beta_i)_{i=1}^m||_{q_1}\, ||\varphi|| \left( \sum_{i=1}^m |\alpha_i|^{p^*} \right)^{1/p^*} \ .
\end{align*}
Taking thereafter the $\sup_{||\varphi|| \leq1}$ and $\sup_{||x|| \leq1}$ in the above inequality, it follows that
\[ ||T|| \left(\sum_{i=1}^m |\lambda_i|^{q_0}\right)^{1/q_0} \leq C\, ||(\beta_i)_{i=1}^m||_{q_1}\,\left( \sum_{i=1}^m |\alpha_i|^{p^*} \right)^{1/p^*}\ .\]
Therefore we conclude that if $T \neq 0$ then $(\lambda_i)_{i=1}^\infty \in l_{q_0}$, which contradicts the assertion (\ref{equa1}).
\end{proof}

\vspace{0,25cm}
Under all conditions and notations given above we will establish the following theorem. For our case, will be sufficient to consider $r=1$ and $t=2$ in Definition \ref{defabst} and Theorem \ref{tdpg}:

\vspace{0,25cm}
\begin{theorem}\label{teoabs}
Let $f:X \rightarrow Y$ be an application belonging to
$\mathcal{H}$ and let
\[
0<q_0,q_1,p_0,p_1,p^{*}<\infty\ ,
\]
such that
\[
\frac{1}{q_0} = \frac{1}{q_1} + \frac{1}{p^{*}%
}\ \ \mathrm{e}\ \ \frac{1}{p_0} = \frac{1}{p_1} + \frac{1}{p^{*}}.
\]
If $(R_1)_{(x,b)}(\cdot)$ is constant, for each $x$ and for each $b$,
then the following statements are equivalent:

\begin{quote}
\noindent $(i)$ $f$ is $R_{1},R_{2}$-$S$-abstract $(q_1,p^{*})$-summing;

\noindent $(ii)$ $f$ is $R_{1},R_{2}$-$S$-abstract $(p_1,p^{*})$-summing.
\end{quote}
\end{theorem}

\begin{proof}
By Theorem \ref{tdpg}, $f$ is $R_{1},R_{2}$-$S$-abstract $(q_1,p^{*})$-summing if and only if there exist a constant $C>0$ and Borel probability measures
$\mu_{i}$ in $K_{i}$, $i=1,2$, such that
\[
S(f,x,b_{1},b_{2}) \leq C \prod_{i=1}^{2} \left(
\int_{K_{i}} R_{i}(\varphi,x,b_{i})^{p_{i}}d\mu_{i}\right)
^{1/p_{i}}\ .
\]
In our case, $f$ is $R_{1},R_{2}$-$S$-abstract $(q_1,p^{*})$-summing if and only if there exist a constant $C>0$ and a Borel probability measure $\mu$ in $K_{2}$ such that
\begin{equation}\label{dominL}
S(f,x,b_{1},b_{2}) \leq C \left(
R_{1}(\varphi,x,b_{1}) \right) \cdot\left(  \int_{K_{2}} R_{2}%
(\varphi,x,b_{2})^{p^{*}}d\mu\right)  ^{1/p^{*}}\ ,
\end{equation}
since, by hypothesis, for any fixed $\varphi \in K_1$,
\[
\left(  \int_{K_{1}} R_{1}(\varphi,x,b_{1})^{q_1}d\mu
_{1}\right)  ^{1/q_1} = R_{1}(\varphi,x,b_{1}%
).
\]
On the other hand, the same reasoning shows that $f$ is $R_{1},R_{2}$%
-$S$-abstract $(p_1,p^{*})$-summing if and only if there exist a constant $C>0$
and a Borel probability measure $\mu$ in $K_{2}$ such that
\[
S(f,x,b_{1},b_{2}) \leq C \left(  R_{1}%
(\varphi,x,b_{1})\right) \cdot\left(  \int_{K_{2}} R_{2}(\varphi
,x,b_{2})^{p^{*}}d\mu\right)  ^{1/p^{*}}\ ,
\]
and this expression corresponds exactly to the one given by (\ref{dominL}).
\end{proof}

The consequence of the above theorem is the coincidence $\mathcal{D}_{p}(E;F) = \mathcal{C}_{(q_0,q_1;p)}(E;F)$, which implies several ways to characterize the class of Cohen strongly summing linear operators. This is shown by the following corollary:

\vspace{0,25cm}
\begin{corollary}\label{coincide}
For all $\left(  q_{0},q_{1}\right)  ,\left(  p_{0}.p_{1}\right)  \in\Gamma$,
\[
\mathcal{C}_{(q_{0},q_{1};p)}(E;F)=\mathcal{C}_{(p_{0},p_{1};p)}(E;F).
\]
In particular,
\[
\mathcal{C}_{(q_0,q_1;p)}(E;F)= \mathcal{C}_{(1,p;p)}(E;F) = \mathcal{D}_{p}(E;F)\ ,
\]
for all $\left(  q_0,q_1\right) \in\Gamma$.
\end{corollary}

\begin{proof}
Let $T \in \mathcal{C}_{(q_{0},q_{1};p)}(E;F)$. Since $\frac{1}{q_0} = \frac{1}{q_1} + \frac{1}{p^{*}}$, with the choices

\[
\left\vert
\begin{array}
[c]{l}%
t=2,\ r=1\\
E_1 = \{0\},\ X_1 = E,\ Y = F\\
f = T,\ \ \mathcal{H}=\mathcal{L}(E;F)\\
K_1 = \{0\}\ \text{and}\ \ K_{2} =B_{F^{^{\prime\prime}}}\\
G_{1}=E,\ \mathrm{and}\ \ G_{2}=F^{^{\prime}}\\
p_{0}=q_0,\ p_{1}=q_1,\ \mathrm{and}\  p_{2}=p^{\ast}\\
S\left(T,0,x,\varphi\right)=\left\vert \varphi\left(T\left(x\right)\right)\right\vert\\
R_{1}\left(\gamma,0,x\right)=\left\Vert x\right\Vert\\
R_{2}(\psi,0,\varphi)=|\psi(\varphi)|\\
\end{array}
\right.
\]
it follows that
\begin{align*}
\left(  \sum_{i=1}^{m}\left(  S\left(T,x_{i},b_{i}^{(1)},b_{i}^{(2)}\right)\right)  ^{p_0}\right)  ^{1/p_0}  &  =\left(
\sum_{i=1}^{m}\left(  S\left(T,0,x_{i},\varphi_{i}\right)\right)  ^{q_{0}}\right)
^{1/q_{0}}\\
&  = \left(\sum_{i=1}^{m}\left\vert \varphi_{i}\left(T\left(x_{i}\right)\right)\right\vert ^{q_{0}}\right)^{1/q_{0}}\ ,
\end{align*}
and also
\begin{align*}
&  \prod_{k=1}^{2}\sup_{\varphi\in K_{k}}\left(  \sum_{i=1}^{m}R_{k}%
\left(\varphi,x_{i},b_{i}^{(k)}\right)^{p_{k}}\right)  ^{1/p_{k}}\\
&  =\sup_{\gamma\in K_{1}}\left(  \sum_{i=1}^{m}R_{1}\left(\gamma
,0,x_{i}\right)^{q_{1}}\right)  ^{1/q_{1}}  \cdot \sup_{\psi\in K_{2}}\left(  \sum_{i=1}^{m}R_{2}(\psi,0,
\varphi_{i})^{p^*}\right)  ^{1/p^*}\\
&  =\left(  \sum_{i=1}^{m}\left\Vert x_{i}\right\Vert^{q_1}\right)  ^{1/q_1}\  \cdot \sup_{\psi\in B_{Y^{^{\prime\prime}}}}\left(  \sum_{i=1}^{m}|\psi\left(
\varphi_{i}\right)  |^{p^{\ast}}\right)  ^{1/{p^{\ast}}}\\
&  =\left\Vert \left(  x_{i}\right)  _{i=1}^{m}\right\Vert _{q_1} \cdot \  ||(\varphi_{i})_{i=1}^{m}||_{w,p^{\ast}}\ .
\end{align*}
Thus, by Definition \ref{defabst}, it follows that $T$ is $R_{1},R_{2}$-$S$-abstract $(q_1,p^{*})$-summing and, by Theorem \ref{teoabs}, $T$ is $R_{1},R_{2}$-$S$-abstract $(p_1,p^{*})$-summing. Therefore there is a constant $C>0$ such that
\[
\left(\sum\limits_{j=1}^{m} \left\vert \varphi_{i}\left(  T\left(  x_{i}\right)  \right)  \right\vert
^{p_0}\right)  ^{1/p_0}\leq C\left\Vert \left(  x_{i}\right)  _{i=1}
^{m}\right\Vert _{p_1}\left\Vert \left(  \varphi_{i}\right)  _{i=1}
^{m}\right\Vert _{w,p^{\ast}}\ ,
\]
for all positive integers $m$ and for all $x_i \in E$, $\varphi_i \in F^{'}$, $i=1,...,m$. So $T \in \mathcal{C}_{(p_{0},p_{1};p)}(E;F)$.

The other inclusion is obtained by the same argument.
\end{proof}

\vspace{0,15cm}
\section{The multilinear case and applications}

\vspace{0,25cm}
In our recent work \cite{jamilson13}, we establish an alternative definition for the concept of Cohen strongly summing multilinear operators given by Achour-Mezrag (Definition \ref{defAchour}):

\vspace{0,25cm}
\begin{theorem}[Campos, \cite{jamilson13}]\label{campos}
Let $1 < p \leq \infty$, with $1/p + 1/p^* = 1$. For $T\in\mathcal{L}(X_{1},...,X_{n};Y)$, the following statements are equivalent:

\begin{quote}
\noindent  $(i)$ There is a constant $C>0$ such that
\begin{equation*}
\sum_{i=1}^{m}\left\vert \varphi_{i}\left(T\left(x_{i}^{(1)},...,x_{i}^{(n)}\right)\right)\right\vert \leq C\left(
\sum_{i=1}^{m}\prod_{j=1}^{n}\left\Vert x_{i}^{(j)}\right\Vert^{p}\right)  ^{1/p}\ ||(\varphi
_{i})_{i=1}^{m}||_{w,p^{\ast}},
\end{equation*}
for all $m\in\mathbb{N},\ x_{i}^{(j)}\in X_{j},\ \varphi_{i}\in Y^{^{\prime}%
},\ i=1,...,m\ ,\ j=1,...,n$ ;

\noindent  $(ii)$ There is a constant $C>0$ such that
\begin{equation*}
\sum_{i=1}^m \left\vert \varphi_i\left(T\left(x_i^{(1)},...,x_i^{(n)}\right)\right)\right\vert
\leq C \left(\sum_{i=1}^m  \left\Vert x_i^{(1)}\right\Vert ^{np}\right)^{1/np} ... \ \left(\sum_{i=1}^m  \left\Vert x_i^{(n)}\right\Vert^{np} \right)^{1/np} \ ||(\varphi_i)_{i=1}^m||_{w,p^*}\ ,
\end{equation*}
for all $m\in\mathbb{N},\ x_{i}^{(j)}\in X_{j},\ \varphi_{i}\in Y^{^{\prime}%
},\ i=1,...,m\ ,\ j=1,...,n$.
\end{quote}
\end{theorem}

\vspace{0,25cm}
As a consequence of Theorem \ref{campos} and the following theorem, similarly to the linear case, we can characterize the class of Cohen strongly summing multilinear operators by means of several inequalities. To do this, we generalize the abstract result of Theorem \ref{teoabs}.

\vspace{0,25cm}
\begin{theorem}\label{teoabsm}
Let $f:X_{1} \times\cdots\times X_{n} \rightarrow Y$ be an application belonging to
$\mathcal{H}$ and let
\[
0<p^{*},p_0,q_0,p_{1},...,p_{t-1},q_{1},...,q_{t-1}<\infty\ ,
\]
such that
\[
\frac{1}{p_0} = \frac{1}{p_{1}} + \cdots+ \frac{1}{p_{t-1}} + \frac{1}{p^{*}%
}\ \ \mathrm{and}\ \ \frac{1}{q_0} = \frac{1}{q_{1}} + \cdots+ \frac{1}{q_{t-1}}
+ \frac{1}{p^{*}}.
\]
If $R_{k_{(x_{1},...,x_{r},b)}}(\cdot)$ is constant for all $x_{1},...,x_{r},b$ and for all $1 \leq k \leq t-1$,
then the following statements are equivalent:

\begin{quote}
\noindent  $(i)$ $f$ is $R_{1},...,R_{t}$-$S$-abstract $(p_{1}
,...,p_{t-1},p^{*})$-summing;

\noindent  $(ii)$ $f$ is $R_{1},...,R_{t}$-$S$-abstract $(q_{1},...,q_{t-1},p^{*})$-summing.
\end{quote}
\end{theorem}

\begin{proof}
Similarly to Theorem \ref{teoabs}, it follows that $f$ is $R_{1},...,R_{t}$-$S$-abstract $(p_{1}%
,...,p_{t-1},p^{*})$-summing if and only if there exist a constant $C>0$ and a Borel probability measure $\mu$
in $K_{t}$ such that
\begin{equation*}
\label{domin}S(f,x_{1},...,x_{r},b_{1},...,b_{t}) \leq C \left( \prod_{i=1}^{t-1}
R_{i}(\varphi,x_{1},...,x_{r},b_{i}) \right) \cdot\left(  \int_{K_{t}} R_{t}%
(\varphi,x_{1},...,x_{r},b_{t})^{p^{*}}d\mu\right)  ^{1/p^{*}}\ .
\end{equation*}
An analogous reasoning shows that $f$ is $R_{1},...,R_{t}$%
-$S$-abstract $(q_{1},...,q_{t-1},p^{*})$-summing if and only if there exist a constant $C>0$
and a Borel probability measure $\mu$ in $K_{t}$ such that
\[
S(f,x_{1},...,x_{r},b_{1},...,b_{t}) \leq C \left( \prod_{i=1}^{t-1} R_{i}%
(\varphi,x_{1},...,x_{r},b_{i})\right) \cdot\left(  \int_{K_{t}} R_{t}(\varphi
,x_{1},...,x_{r},b_{t})^{p^{*}}d\mu\right)  ^{1/p^{*}}\ .
\]
\end{proof}

\vspace{0,25cm}
\begin{corollary}\label{teo2}
For $T \in \mathcal{L}(E_1,...,E_n;F)$, $1 = \frac{1}{p} + \frac{1}{p^*}$ and $\frac{1}{q_0} = \frac{1}{q_1} + \cdots + \frac{1}{q_n} + \frac{1}{p^*}$, the following statements are equivalent:
\begin{quote}
\noindent  $(i)$ There exists a constant $C>0$ such that
\begin{equation*}
\sum_{i=1}^m \left\vert \varphi_i\left(T\left(x_i^{(1)},...,x_i^{(n)}\right)\right)\right\vert
\leq C \left(\sum_{i=1}^m  \left\Vert x_i^{(1)}\right\Vert ^{np}\right)^{1/np} ... \ \left(\sum_{i=1}^m  \left\Vert x_i^{(n)}\right\Vert^{np} \right)^{1/np} \ ||(\varphi_i)_{i=1}^m||_{w,p^*}\ ,
\end{equation*}
for all $m \in \mathbb{N},\ x_i^{(j)} \in E_j,\ \varphi_i \in F^{'},\ i=1,...,m\ , \ j=1,...,n$;

\noindent  $(ii)$ There exists a constant $C>0$ such that
\begin{align}\label{defq}
& \left(\sum_{i=1}^m \left\vert \varphi_i\left(T\left(x_i^{(1)},...,x_i^{(n)}\right)\right)\right\vert^{q_0}\right)^{1/q_0} \nonumber \\
& \leq C \left(\sum_{i=1}^m  \left\Vert x_i^{(1)}\right\Vert ^{q_1} \right)^{1/q_1} ... \ \left(\sum_{i=1}^m  \left\Vert x_i^{(n)}\right\Vert ^{q_n} \right)^{1/q_n} \ ||(\varphi_i)_{i=1}^m||_{w,p^*}\ ,
\end{align}
for all $m \in \mathbb{N},\ x_i^{(j)} \in E_j,\ \varphi_i \in F^{'},\ i=1,...,m\ , \ j=1,...,n$ .
\end{quote}
\end{corollary}

\begin{proof}
With adequate choices and using Definition \ref{defabst} (see \cite{jamilson13}, Theorem 3.8), if $T$ satisfies $(i)$ then $T$ is  $R_{1},...,R_{n+1}$-$S$-abstract $(np,,...,np,p^{*})$-summing. Using the same argument it can be shown that if $T$ satisfies $(ii)$ then $T$ is $R_{1},...,R_{n+1}$-$S$-abstract $(q_{1},...,q_{n},p^{*})$-summing. So, by Theorem \ref{teoabsm}, we finish the proof.
\end{proof}

\vspace{0,25cm}
From Corollary \ref{teo2} it is also possible to establish more characterizations for the class of Cohen strongly summing multilinear operators by means of sequences and, equivalently, by means of an inequality similar to that given by (\ref{defq}), but adding the index $i$ to infinity:

\vspace{0,25cm}
Let $1 < p \leq \infty$, $E_j,F$ be Banach spaces, $j=1,...,n$ and $\frac{1}{q_0} = \frac{1}{q_1} + \cdots + \frac{1}{q_n} + \frac{1}{p^*}$. An operator $T \in \mathcal{L}(E_1,...,E_n;F)$ is Cohen strongly $p$-summing if
\[ \left(\varphi_i\left( T\left(x_i^{(1)},...,x_i^{(n)}\right)\right)\right)_{i=1}^\infty \in l_{q_0}\ \ \mathrm{whenever}\ \ \left(x_i^{(j)}\right)_{i=1}^\infty \in l_{q_j}(E_j) \ ,\ j=1,...,n\ .\]

\vspace{0,25cm}
This last definition amounts to saying that an operator $T \in \mathcal{L}(E_1,...,E_n;F)$ is Cohen strongly $p$-summing if There exists a constant $C>0$ such that
\begin{equation*}
\left(\sum_{i=1}^\infty \left\vert \varphi_i\left(T\left(x_i^{(1)},...,x_i^{(n)}\right)\right)\right\vert^{q_0}\right)^{1/q_0}
\leq C \left(\sum_{i=1}^\infty  \left\Vert x_i^{(1)}\right\Vert ^{q_1} \right)^{1/q_1} ... \ \left(\sum_{i=1}^\infty  \left\Vert x_i^{(n)}\right\Vert ^{q_n} \right)^{1/q_n} \ ||(\varphi_i)_{i=1}^\infty||_{w,p^*}\ ,
\end{equation*}
whenever $\left(x_i^{(j)}\right)_{i=1}^\infty \in l_{q_j}(E_j) \ ,\ j=1,...,n\ $ and $(\varphi_i)_{i=1}^\infty \in l_{p^*}^w(F^{'})$.

\vspace{0,25cm}
These kinds of characterizations for the class of Cohen strongly summing multilinear operators appear in our work \cite{jamilson13} but without this abstract focus.

\end{document}